\documentclass[11pt]{article}
\usepackage[T1]{fontenc}
\usepackage[utf8]{inputenc}
\usepackage{lmodern}
\usepackage{amsmath,amssymb,amsthm}
\usepackage{geometry}
\usepackage{graphicx}
\usepackage{tikz}
\usetikzlibrary{calc}
\usepackage{hyperref}
\usepackage{booktabs}
\usepackage{enumitem}
\usepackage{caption}
\usepackage{xcolor}
\usepackage{float}
\usepackage{accsupp}
\usepackage[percent]{overpic}
\title{Wheel Classes in Kontsevich Graph Complex and Merkulov's Low-Valence Conjecture}
\author{Assar Andersson}
\date{May 20, 2026}

\hypersetup{
  colorlinks=true,
  linkcolor=blue!50!black,
  citecolor=blue!50!black,
  urlcolor=blue!50!black
}

\theoremstyle{definition}
\newtheorem{definition}{Definition}[section]

\theoremstyle{plain}
\newtheorem{theorem}[definition]{Theorem}
\newtheorem*{maintheorem}{Theorem}
\newtheorem{proposition}[definition]{Proposition}
\newtheorem{lemma}[definition]{Lemma}

\newtheorem{conjecture}[definition]{Conjecture}

\newcommand{\LeftChoice}{\mathrm{left}}
\newcommand{\RightChoice}{\mathrm{right}}
\newcommand{\GC}{\mathsf{GC}}

\begin{document}
\maketitle

\begin{abstract}
We show that the wheel classes in the Kontsevich graph complex $\GC_d$ admit representatives supported on graphs with only
$3$- and $4$-valent vertices. This verifies that Merkulov’s low-valence conjecture holds for the wheel classes.
 More precisely, for every $m \ge 2$, we prove that the wheel graph $W_{2m+1}$ is homologous to an explicit linear combination of $2^{m-2}$ graphs, each having only $3$- and $4$-valent vertices.
\end{abstract}

\medskip
\noindent\textbf{Acknowledgments.} The author is grateful to Sergei Merkulov for helpful comments on an earlier version of this paper.

\section{Introduction}
A conjecture of Merkulov predicts that every homology class in the Kontsevich graph complex $\GC_d$ should admit a representative supported on graphs whose vertices all have valence $3$ or $4$.

\begin{conjecture}[Merkulov]\label{conj:merkulov}
Let $\GC_d^{\mathrm{Me}}$ be the maximal subcomplex of $\GC_d$ spanned by linear combinations of graphs all of whose vertices have valence $3$ or $4$.
Then the inclusion
\[
i \colon \GC_d^{\mathrm{Me}}\hookrightarrow \GC_d
\]
is a quasi-isomorphism.
\end{conjecture}

This conjecture has been checked by Brun and Willwacher up to genus $10$ in all cohomological degrees \cite{BrunWillwacher2024}.

In this paper we show that the homology classes in the even graph complex represented by the wheel graphs
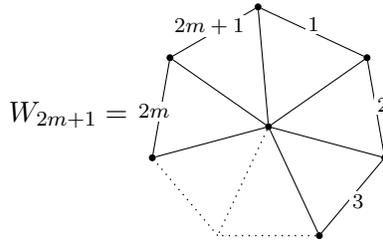
\[
W_{2m+1} =
\begin{tikzpicture}[x=0.78cm,y=0.78cm,baseline=(current bounding box.center)]
  \tikzset{graphvertex/.style={circle,fill=black,draw=black,inner sep=0pt,minimum size=2.2pt}}
  \tikzset{graphedge/.style={draw=black,line width=0.4pt}}
  \tikzset{edgelabel/.style={font=\scriptsize,inner sep=1pt,fill=white,text=black,execute at begin node=\BeginAccSupp{ActualText={}},execute at end node=\EndAccSupp{}}}

  \node[graphvertex] (c) at (0,0) {};
  \coordinate (t) at (-115:2.05);
  \node[graphvertex] (v1) at (35:2.05) {};
  \node[graphvertex] (v2) at (-15:2.05) {};
  \node[graphvertex] (v3) at (-65:2.05) {};
  \node[graphvertex] (v4) at (-165:2.05) {};
  \node[graphvertex] (v5) at (145:2.05) {};
  \node[graphvertex] (v6) at (95:2.05) {};

  \draw[dotted, line width=0.5pt] (c) -- (t);
  \draw[graphedge] (c) -- (v1);
  \draw[graphedge] (c) -- (v2);
  \draw[graphedge] (c) -- (v3);
  \draw[graphedge] (c) -- (v4);
  \draw[graphedge] (c) -- (v5);
  \draw[graphedge] (c) -- (v6);

  \draw[graphedge] (v6) -- (v1);
  \node[edgelabel] at ($(v6)!0.5!(v1)+(0.02,0.14)$) {1};
  \draw[graphedge] (v1) -- (v2);
  \node[edgelabel] at ($(v1)!0.5!(v2)+(0.12,0.05)$) {2};
  \draw[graphedge] (v2) -- (v3);
  \node[edgelabel] at ($(v2)!0.5!(v3)+(0.10,-0.08)$) {3};
  \draw[dotted, line width=0.5pt] (v3) -- (t);
  \draw[dotted, line width=0.5pt] (t) -- (v4);
  \draw[graphedge] (v4) -- (v5);
  \node[edgelabel] at ($(v4)!0.5!(v5)+(-0.12,-0.06)$) {\(2m\)};
  \draw[graphedge] (v5) -- (v6);
  \node[edgelabel] at ($(v5)!0.5!(v6)+(-0.08,0.10)$) {\(2m+1\)};
\end{tikzpicture}
\]
are consistent with Merkulov’s conjecture.

These classes occupy a distinguished place in $\GC_d$ for $d$ even. By Willwacher's identification of the degree-zero cohomology of 
$\GC_2$ with the Grothendieck--Teichm\"uller Lie algebra \cite{Willwacher2015}, 
they correspond to the conjectural generators of $\mathfrak{grt}_1$; see also \cite{MerkulovSurvey}.

The wheel graphs $W_{2m+1}$ are obviously not contained in $\GC_d^{\mathrm{Me}}$ for $m>1$, 
since they contain one $2m+1$-valent hub vertex. 
Our main theorem shows that the homology classes associated with the wheel graphs nevertheless admit representatives in $\GC_d^{\mathrm{Me}}$.

\begin{maintheorem}
For $d$ even there exists an element $U_{2m+1} \in \GC_d$ such that
\[
W_{2m+1} -d(U_{2m+1})\in \GC_d^{\mathrm{Me}}.
\]
\end{maintheorem}

An explicit version, with a formula for $U_{2m+1}$, is stated and proved later as Theorem~\ref{thm:explicit-representatives}.

\section{Conventions}
We shall assume familiarity with Kontsevich's graph complex in this paper.
To fix conventions, we will consider $\GC_d$ to consist of connected graphs in which every vertex is at least $3$-valent.

We will use vertex and edge labels starting from $0$.
We will exclusively consider the case that $d$ is even. In this case, vertex labels and edge directions do not matter, and edge labels matter only through the sign orientation.

The differential $d$ is the edge-contraction differential. If a graph has $m$ edges, then the contraction of the edge labelled $i$ carries the sign
\[
(-1)^{i+m}.
\]
After contracting the edge labelled $i$, we relabel the remaining edges by preserving their relative order: labels smaller than $i$ remain unchanged, while each label greater than $i$ is decreased by $1$.

\section{Two Families of Graphs}
The proof is organized around two explicit families of graphs.
The family $V_N(S)$ describes the low-valence support that remains after eliminating the higher-valence part of the wheel class, while the family $U_N(S)$ provides the auxiliary chains whose differentials perform this elimination.
The construction of $V_N(S)$ and $U_N(S)$ may be summarized as follows.

\medskip
\noindent
\textbf{Construction of $V_N(S)$ and $U_N(S)$.}
\begin{enumerate}[itemsep=0.35em, topsep=0.35em]
\item Input: an odd integer $N$ and a binary sequence $S=(s_1,\dots,s_k)$ of choices, with $s_i \in \{\LeftChoice, \RightChoice\}$. 

\item \label{it:seed} Initialize the graph with edges
\[
(0,1),\ (0,2),\ (0,3),\ (1,2),\ (2,3). 
\]
Denote vertex $1$ by L, vertex $2$ by C and vertex $3$ by R.

\item \label{it:iterate} For $i=1,\dots,k$:
\begin{enumerate}[label=(\alph*), itemsep=0.25em, topsep=0.25em]
\item \label{it:spine} Add a new vertex attached to the current centre vertex $C$, and denote the new vertex by $\texttt{new\_C}$. The new edge $(\texttt{C},\texttt{new\_C})$ gets the label $4i+1$. 

\item \label{it:triangle1} Add an edge $(\texttt{new\_C},R)$ (if $s_i=\LeftChoice$, we instead add the edge $(\texttt{new\_C},L)$). This edge gets the label $4i+2$.
\item \label{it:triangle2} Add a second new vertex denoted $\texttt{new\_R}$ (or $\texttt{new\_L}$). Add an edge $(\texttt{R},\texttt{new\_R})$ (or $(\texttt{L},\texttt{new\_L})$). This new edge gets the label $4i+3$.
\item \label{it:triangle3} Add an edge $(\texttt{new\_C},\texttt{new\_R})$ (or $(\texttt{new\_C},\texttt{new\_L})$). This edge gets the label $4i+4$.
\item Denote the vertices $\texttt{new\_C}$ and $\texttt{new\_R}$ (or $\texttt{new\_L}$) by $C$ and $R$ (or $L$) respectively. 

\end{enumerate}

After these $k$ steps, the graph has $2k+4$ vertices and $4k+5$ edges.  

\item \label{it:U}(Only for $U_N(S)$) 
Add one further centre vertex $\texttt{new\_C}$ and an edge $(\texttt{C},\texttt{new\_C})$. This edge gets the label $4k+5$.
Denote the new vertex $C$.

\item \label{it:arc} 
Now, for both $V_N(S)$ and $U_N(S)$, let $m$ denote the number of remaining vertices to be added; explicitly, $m=N-2k-3$. Let $n$ denote the current largest vertex label.
If $m = 0$, add the final edge $(L,R)$, with edge label $2N-1$ (this case will only be considered for $V_N(S)$). 
Otherwise, add the remaining vertices and edges as an arc of triangles as follows: 
Let us denote the least unassigned edge label $e = 4k+5$ for $V_N(S)$ and $e = 4k+6$ for $U_N(S)$.

\begin{enumerate}
\item \label{it:left_arc} Add $(L, n+1)$ with edge label $e$, and add $(n+1,C)$ with edge label $e+1$.
\item \label{it:mid_arc} For $i = 1, \ldots, m-1$: Add $(n+i, n+i+1)$ with edge label $e+2i$, and add $(n+i+1,C)$ with edge label $e+2i+1$.
\item \label{it:right_arc} Add $(n+m, R)$ with edge label $2N-1$ for $V_N(S)$ and $2N$ for $U_N(S)$. 

\end{enumerate}
\end{enumerate}
\medskip

The construction is illustrated in the following two examples, shown in Figures~\ref{fig:V11-construction} and~\ref{fig:V15-construction}.

\begin{figure}[H]
\centering
\begin{minipage}{0.24\linewidth}
\centering
\resizebox{\linewidth}{!}{\begin{tikzpicture}[x=0.8cm,y=0.8cm,baseline=(current bounding box.center)]
  \tikzset{graphvertex/.style={circle,fill=black,draw=black,inner sep=0pt,minimum size=2.2pt}}
  \tikzset{graphedge/.style={draw=black,line width=0.4pt}}
  \tikzset{edgelabel/.style={font=\scriptsize,inner sep=1pt,fill=white,text=black,execute at begin node=\BeginAccSupp{ActualText={}},execute at end node=\EndAccSupp{}}}
  \node[graphvertex] (v0) at (0.000,2.200) {};
  \node[graphvertex] (v1) at (-3.000,0.000) {};
  \node[graphvertex] (v2) at (0.000,0.000) {};
  \node[graphvertex] (v3) at (3.000,0.000) {};
  \draw[graphedge] (v0) -- (v1);
  \node[edgelabel] at (-1.405,0.971) {0};
  \draw[graphedge] (v0) -- (v2);
  \node[edgelabel] at (0.160,1.100) {1};
  \draw[graphedge] (v0) -- (v3);
  \node[edgelabel] at (1.595,1.229) {2};
  \draw[graphedge] (v1) -- (v2);
  \node[edgelabel] at (-1.500,0.160) {3};
  \draw[graphedge] (v2) -- (v3);
  \node[edgelabel] at (1.500,0.160) {4};
\end{tikzpicture}}
\caption*{Seed}
\end{minipage}\hfill
\begin{minipage}{0.24\linewidth}
\centering
\resizebox{\linewidth}{!}{\begin{tikzpicture}[x=0.8cm,y=0.8cm,baseline=(current bounding box.center)]
  \tikzset{graphvertex/.style={circle,fill=black,draw=black,inner sep=0pt,minimum size=2.2pt}}
  \tikzset{graphedge/.style={draw=black,line width=0.4pt}}
  \tikzset{edgelabel/.style={font=\scriptsize,inner sep=1pt,fill=white,text=black,execute at begin node=\BeginAccSupp{ActualText={}},execute at end node=\EndAccSupp{}}}
  \node[graphvertex] (v0) at (0.000,2.200) {};
  \node[graphvertex] (v1) at (-3.000,0.000) {};
  \node[graphvertex] (v2) at (0.000,0.000) {};
  \node[graphvertex] (v3) at (3.000,0.000) {};
  \node[graphvertex] (v4) at (0.000,-2.200) {};
  \node[graphvertex] (v5) at (-3.000,-2.200) {};
  \draw[graphedge] (v0) -- (v1);
  \node[edgelabel] at (-1.405,0.971) {0};
  \draw[graphedge] (v0) -- (v2);
  \node[edgelabel] at (0.160,1.100) {1};
  \draw[graphedge] (v0) -- (v3);
  \node[edgelabel] at (1.595,1.229) {2};
  \draw[graphedge] (v1) -- (v2);
  \node[edgelabel] at (-1.500,0.160) {3};
  \draw[graphedge] (v2) -- (v3);
  \node[edgelabel] at (1.500,0.160) {4};
  \draw[graphedge] (v2) -- (v4);
  \node[edgelabel] at (0.160,-1.100) {5};
  \draw[graphedge] (v4) -- (v1);
  \node[edgelabel] at (-1.595,-1.229) {6};
  \draw[graphedge] (v1) -- (v5);
  \node[edgelabel] at (-2.840,-1.100) {7};
  \draw[graphedge] (v4) -- (v5);
  \node[edgelabel] at (-1.500,-2.360) {8};
\end{tikzpicture}}
\caption*{After $s_1=\LeftChoice$}
\end{minipage}\hfill
\begin{minipage}{0.24\linewidth}
\centering
\resizebox{\linewidth}{!}{\begin{tikzpicture}[x=0.8cm,y=0.8cm,baseline=(current bounding box.center)]
  \tikzset{graphvertex/.style={circle,fill=black,draw=black,inner sep=0pt,minimum size=2.2pt}}
  \tikzset{graphedge/.style={draw=black,line width=0.4pt}}
  \tikzset{edgelabel/.style={font=\scriptsize,inner sep=1pt,fill=white,text=black,execute at begin node=\BeginAccSupp{ActualText={}},execute at end node=\EndAccSupp{}}}
  \node[graphvertex] (v0) at (0.000,2.200) {};
  \node[graphvertex] (v1) at (-3.000,0.000) {};
  \node[graphvertex] (v2) at (0.000,0.000) {};
  \node[graphvertex] (v3) at (3.000,0.000) {};
  \node[graphvertex] (v4) at (0.000,-2.200) {};
  \node[graphvertex] (v5) at (-3.000,-2.200) {};
  \node[graphvertex] (v6) at (0.000,-4.400) {};
  \node[graphvertex] (v7) at (3.000,-4.400) {};
  \draw[graphedge] (v0) -- (v1);
  \node[edgelabel] at (-1.405,0.971) {0};
  \draw[graphedge] (v0) -- (v2);
  \node[edgelabel] at (0.160,1.100) {1};
  \draw[graphedge] (v0) -- (v3);
  \node[edgelabel] at (1.595,1.229) {2};
  \draw[graphedge] (v1) -- (v2);
  \node[edgelabel] at (-1.500,0.160) {3};
  \draw[graphedge] (v2) -- (v3);
  \node[edgelabel] at (1.500,0.160) {4};
  \draw[graphedge] (v2) -- (v4);
  \node[edgelabel] at (0.160,-1.100) {5};
  \draw[graphedge] (v4) -- (v1);
  \node[edgelabel] at (-1.595,-1.229) {6};
  \draw[graphedge] (v1) -- (v5);
  \node[edgelabel] at (-2.840,-1.100) {7};
  \draw[graphedge] (v4) -- (v5);
  \node[edgelabel] at (-1.500,-2.360) {8};
  \draw[graphedge] (v4) -- (v6);
  \node[edgelabel] at (0.160,-3.300) {9};
  \draw[graphedge] (v6) -- (v3);
  \node[edgelabel] at (1.368,-2.110) {10};
  \draw[graphedge] (v3) -- (v7);
  \node[edgelabel] at (3.160,-2.200) {11};
  \draw[graphedge] (v6) -- (v7);
  \node[edgelabel] at (1.500,-4.240) {12};
\end{tikzpicture}}
\caption*{After $s_2=\RightChoice$}
\end{minipage}\hfill
\begin{minipage}{0.24\linewidth}
\centering
\resizebox{\linewidth}{!}{\begin{tikzpicture}[x=0.8cm,y=0.8cm,baseline=(current bounding box.center)]
  \tikzset{graphvertex/.style={circle,fill=black,draw=black,inner sep=0pt,minimum size=2.2pt}}
  \tikzset{graphedge/.style={draw=black,line width=0.4pt}}
  \tikzset{edgelabel/.style={font=\scriptsize,inner sep=1pt,fill=white,text=black,execute at begin node=\BeginAccSupp{ActualText={}},execute at end node=\EndAccSupp{}}}
  \node[graphvertex] (v0) at (0.000,2.200) {};
  \node[graphvertex] (v1) at (-3.000,0.000) {};
  \node[graphvertex] (v2) at (0.000,0.000) {};
  \node[graphvertex] (v3) at (3.000,0.000) {};
  \node[graphvertex] (v4) at (0.000,-2.200) {};
  \node[graphvertex] (v5) at (-3.000,-2.200) {};
  \node[graphvertex] (v6) at (0.000,-4.400) {};
  \node[graphvertex] (v7) at (3.000,-4.400) {};
  \node[graphvertex] (v8) at (-3.074,-4.173) {};
  \node[graphvertex] (v9) at (-1.973,-5.813) {};
  \node[graphvertex] (v10) at (-0.119,-6.493) {};
  \node[graphvertex] (v11) at (1.780,-5.953) {};
  \draw[graphedge] (v0) -- (v1);
  \node[edgelabel] at (-1.405,0.971) {0};
  \draw[graphedge] (v0) -- (v2);
  \node[edgelabel] at (0.160,1.100) {1};
  \draw[graphedge] (v0) -- (v3);
  \node[edgelabel] at (1.595,1.229) {2};
  \draw[graphedge] (v1) -- (v2);
  \node[edgelabel] at (-1.500,0.160) {3};
  \draw[graphedge] (v2) -- (v3);
  \node[edgelabel] at (1.500,0.160) {4};
  \draw[graphedge] (v2) -- (v4);
  \node[edgelabel] at (0.160,-1.100) {5};
  \draw[graphedge] (v4) -- (v1);
  \node[edgelabel] at (-1.595,-1.229) {6};
  \draw[graphedge] (v1) -- (v5);
  \node[edgelabel] at (-2.840,-1.100) {7};
  \draw[graphedge] (v4) -- (v5);
  \node[edgelabel] at (-1.500,-2.360) {8};
  \draw[graphedge] (v4) -- (v6);
  \node[edgelabel] at (0.160,-3.300) {9};
  \draw[graphedge] (v6) -- (v3);
  \node[edgelabel] at (1.368,-2.110) {10};
  \draw[graphedge] (v3) -- (v7);
  \node[edgelabel] at (3.160,-2.200) {11};
  \draw[graphedge] (v6) -- (v7);
  \node[edgelabel] at (1.500,-4.240) {12};
  \draw[graphedge] (v5) -- (v8);
  \node[edgelabel] at (-2.877,-3.193) {13};
  \draw[graphedge] (v6) -- (v8);
  \node[edgelabel] at (-1.549,-4.446) {14};
  \draw[graphedge] (v8) -- (v9);
  \node[edgelabel] at (-2.391,-4.904) {15};
  \draw[graphedge] (v6) -- (v9);
  \node[edgelabel] at (-0.893,-5.237) {16};
  \draw[graphedge] (v9) -- (v10);
  \node[edgelabel] at (-0.991,-6.003) {17};
  \draw[graphedge] (v6) -- (v10);
  \node[edgelabel] at (0.100,-5.456) {18};
  \draw[graphedge] (v10) -- (v11);
  \node[edgelabel] at (0.787,-6.069) {19};
  \draw[graphedge] (v6) -- (v11);
  \node[edgelabel] at (0.995,-5.056) {20};
  \draw[graphedge] (v11) -- (v7);
  \node[edgelabel] at (2.264,-5.078) {21};
\end{tikzpicture}}
\caption*{Close the arc}
\end{minipage}
\caption{Step-by-step construction of $V_{11}(\LeftChoice,\RightChoice)$.}
\label{fig:V11-construction}
\end{figure}

\begin{figure}[H]
\centering
\begin{minipage}{0.19\linewidth}
\centering
\resizebox{\linewidth}{!}{\begin{tikzpicture}[x=0.8cm,y=0.8cm,baseline=(current bounding box.center)]
  \tikzset{graphvertex/.style={circle,fill=black,draw=black,inner sep=0pt,minimum size=2.2pt}}
  \tikzset{graphedge/.style={draw=black,line width=0.4pt}}
  \tikzset{edgelabel/.style={font=\scriptsize,inner sep=1pt,fill=white,text=black,execute at begin node=\BeginAccSupp{ActualText={}},execute at end node=\EndAccSupp{}}}
  \node[graphvertex] (v0) at (0.000,2.200) {};
  \node[graphvertex] (v1) at (-3.000,0.000) {};
  \node[graphvertex] (v2) at (0.000,0.000) {};
  \node[graphvertex] (v3) at (3.000,0.000) {};
  \draw[graphedge] (v0) -- (v1);
  \node[edgelabel] at (-1.405,0.971) {0};
  \draw[graphedge] (v0) -- (v2);
  \node[edgelabel] at (0.160,1.100) {1};
  \draw[graphedge] (v0) -- (v3);
  \node[edgelabel] at (1.595,1.229) {2};
  \draw[graphedge] (v1) -- (v2);
  \node[edgelabel] at (-1.500,0.160) {3};
  \draw[graphedge] (v2) -- (v3);
  \node[edgelabel] at (1.500,0.160) {4};
\end{tikzpicture}}
\caption*{Seed}
\end{minipage}\hfill
\begin{minipage}{0.19\linewidth}
\centering
\resizebox{\linewidth}{!}{\begin{tikzpicture}[x=0.8cm,y=0.8cm,baseline=(current bounding box.center)]
  \tikzset{graphvertex/.style={circle,fill=black,draw=black,inner sep=0pt,minimum size=2.2pt}}
  \tikzset{graphedge/.style={draw=black,line width=0.4pt}}
  \tikzset{edgelabel/.style={font=\scriptsize,inner sep=1pt,fill=white,text=black,execute at begin node=\BeginAccSupp{ActualText={}},execute at end node=\EndAccSupp{}}}
  \node[graphvertex] (v0) at (0.000,2.200) {};
  \node[graphvertex] (v1) at (-3.000,0.000) {};
  \node[graphvertex] (v2) at (0.000,0.000) {};
  \node[graphvertex] (v3) at (3.000,0.000) {};
  \node[graphvertex] (v4) at (0.000,-2.200) {};
  \node[graphvertex] (v5) at (-3.000,-2.200) {};
  \draw[graphedge] (v0) -- (v1);
  \node[edgelabel] at (-1.405,0.971) {0};
  \draw[graphedge] (v0) -- (v2);
  \node[edgelabel] at (0.160,1.100) {1};
  \draw[graphedge] (v0) -- (v3);
  \node[edgelabel] at (1.595,1.229) {2};
  \draw[graphedge] (v1) -- (v2);
  \node[edgelabel] at (-1.500,0.160) {3};
  \draw[graphedge] (v2) -- (v3);
  \node[edgelabel] at (1.500,0.160) {4};
  \draw[graphedge] (v2) -- (v4);
  \node[edgelabel] at (0.160,-1.100) {5};
  \draw[graphedge] (v4) -- (v1);
  \node[edgelabel] at (-1.595,-1.229) {6};
  \draw[graphedge] (v1) -- (v5);
  \node[edgelabel] at (-2.840,-1.100) {7};
  \draw[graphedge] (v4) -- (v5);
  \node[edgelabel] at (-1.500,-2.360) {8};
\end{tikzpicture}}
\caption*{After $s_1=\LeftChoice$}
\end{minipage}\hfill
\begin{minipage}{0.19\linewidth}
\centering
\resizebox{\linewidth}{!}{\begin{tikzpicture}[x=0.8cm,y=0.8cm,baseline=(current bounding box.center)]
  \tikzset{graphvertex/.style={circle,fill=black,draw=black,inner sep=0pt,minimum size=2.2pt}}
  \tikzset{graphedge/.style={draw=black,line width=0.4pt}}
  \tikzset{edgelabel/.style={font=\scriptsize,inner sep=1pt,fill=white,text=black,execute at begin node=\BeginAccSupp{ActualText={}},execute at end node=\EndAccSupp{}}}
  \node[graphvertex] (v0) at (0.000,2.200) {};
  \node[graphvertex] (v1) at (-3.000,0.000) {};
  \node[graphvertex] (v2) at (0.000,0.000) {};
  \node[graphvertex] (v3) at (3.000,0.000) {};
  \node[graphvertex] (v4) at (0.000,-2.200) {};
  \node[graphvertex] (v5) at (-3.000,-2.200) {};
  \node[graphvertex] (v6) at (0.000,-4.400) {};
  \node[graphvertex] (v7) at (-3.000,-4.400) {};
  \draw[graphedge] (v0) -- (v1);
  \node[edgelabel] at (-1.405,0.971) {0};
  \draw[graphedge] (v0) -- (v2);
  \node[edgelabel] at (0.160,1.100) {1};
  \draw[graphedge] (v0) -- (v3);
  \node[edgelabel] at (1.595,1.229) {2};
  \draw[graphedge] (v1) -- (v2);
  \node[edgelabel] at (-1.500,0.160) {3};
  \draw[graphedge] (v2) -- (v3);
  \node[edgelabel] at (1.500,0.160) {4};
  \draw[graphedge] (v2) -- (v4);
  \node[edgelabel] at (0.160,-1.100) {5};
  \draw[graphedge] (v4) -- (v1);
  \node[edgelabel] at (-1.595,-1.229) {6};
  \draw[graphedge] (v1) -- (v5);
  \node[edgelabel] at (-2.840,-1.100) {7};
  \draw[graphedge] (v4) -- (v5);
  \node[edgelabel] at (-1.500,-2.360) {8};
  \draw[graphedge] (v4) -- (v6);
  \node[edgelabel] at (0.160,-3.300) {9};
  \draw[graphedge] (v6) -- (v5);
  \node[edgelabel] at (-1.595,-3.429) {10};
  \draw[graphedge] (v5) -- (v7);
  \node[edgelabel] at (-2.840,-3.300) {11};
  \draw[graphedge] (v6) -- (v7);
  \node[edgelabel] at (-1.500,-4.560) {12};
\end{tikzpicture}}
\caption*{After $s_2=\LeftChoice$}
\end{minipage}\hfill
\begin{minipage}{0.19\linewidth}
\centering
\resizebox{\linewidth}{!}{\begin{tikzpicture}[x=0.8cm,y=0.8cm,baseline=(current bounding box.center)]
  \tikzset{graphvertex/.style={circle,fill=black,draw=black,inner sep=0pt,minimum size=2.2pt}}
  \tikzset{graphedge/.style={draw=black,line width=0.4pt}}
  \tikzset{edgelabel/.style={font=\scriptsize,inner sep=1pt,fill=white,text=black,execute at begin node=\BeginAccSupp{ActualText={}},execute at end node=\EndAccSupp{}}}
  \node[graphvertex] (v0) at (0.000,2.200) {};
  \node[graphvertex] (v1) at (-3.000,0.000) {};
  \node[graphvertex] (v2) at (0.000,0.000) {};
  \node[graphvertex] (v3) at (3.000,0.000) {};
  \node[graphvertex] (v4) at (0.000,-2.200) {};
  \node[graphvertex] (v5) at (-3.000,-2.200) {};
  \node[graphvertex] (v6) at (0.000,-4.400) {};
  \node[graphvertex] (v7) at (-3.000,-4.400) {};
  \node[graphvertex] (v8) at (0.000,-6.600) {};
  \node[graphvertex] (v9) at (3.000,-6.600) {};
  \draw[graphedge] (v0) -- (v1);
  \node[edgelabel] at (-1.405,0.971) {0};
  \draw[graphedge] (v0) -- (v2);
  \node[edgelabel] at (0.160,1.100) {1};
  \draw[graphedge] (v0) -- (v3);
  \node[edgelabel] at (1.595,1.229) {2};
  \draw[graphedge] (v1) -- (v2);
  \node[edgelabel] at (-1.500,0.160) {3};
  \draw[graphedge] (v2) -- (v3);
  \node[edgelabel] at (1.500,0.160) {4};
  \draw[graphedge] (v2) -- (v4);
  \node[edgelabel] at (0.160,-1.100) {5};
  \draw[graphedge] (v4) -- (v1);
  \node[edgelabel] at (-1.595,-1.229) {6};
  \draw[graphedge] (v1) -- (v5);
  \node[edgelabel] at (-2.840,-1.100) {7};
  \draw[graphedge] (v4) -- (v5);
  \node[edgelabel] at (-1.500,-2.360) {8};
  \draw[graphedge] (v4) -- (v6);
  \node[edgelabel] at (0.160,-3.300) {9};
  \draw[graphedge] (v6) -- (v5);
  \node[edgelabel] at (-1.595,-3.429) {10};
  \draw[graphedge] (v5) -- (v7);
  \node[edgelabel] at (-2.840,-3.300) {11};
  \draw[graphedge] (v6) -- (v7);
  \node[edgelabel] at (-1.500,-4.560) {12};
  \draw[graphedge] (v6) -- (v8);
  \node[edgelabel] at (0.160,-5.500) {13};
  \draw[graphedge] (v8) -- (v3);
  \node[edgelabel] at (1.354,-3.234) {14};
  \draw[graphedge] (v3) -- (v9);
  \node[edgelabel] at (3.160,-3.300) {15};
  \draw[graphedge] (v8) -- (v9);
  \node[edgelabel] at (1.500,-6.440) {16};
\end{tikzpicture}}
\caption*{After $s_3=\RightChoice$}
\end{minipage}\hfill
\begin{minipage}{0.19\linewidth}
\centering
\resizebox{\linewidth}{!}{\begin{tikzpicture}[x=0.8cm,y=0.8cm,baseline=(current bounding box.center)]
  \tikzset{graphvertex/.style={circle,fill=black,draw=black,inner sep=0pt,minimum size=2.2pt}}
  \tikzset{graphedge/.style={draw=black,line width=0.4pt}}
  \tikzset{edgelabel/.style={font=\scriptsize,inner sep=1pt,fill=white,text=black,execute at begin node=\BeginAccSupp{ActualText={}},execute at end node=\EndAccSupp{}}}
  \node[graphvertex] (v0) at (0.000,2.200) {};
  \node[graphvertex] (v1) at (-3.000,0.000) {};
  \node[graphvertex] (v2) at (0.000,0.000) {};
  \node[graphvertex] (v3) at (3.000,0.000) {};
  \node[graphvertex] (v4) at (0.000,-2.200) {};
  \node[graphvertex] (v5) at (-3.000,-2.200) {};
  \node[graphvertex] (v6) at (0.000,-4.400) {};
  \node[graphvertex] (v7) at (-3.000,-4.400) {};
  \node[graphvertex] (v8) at (0.000,-6.600) {};
  \node[graphvertex] (v9) at (3.000,-6.600) {};
  \node[graphvertex] (v10) at (-3.180,-5.811) {};
  \node[graphvertex] (v11) at (-2.730,-7.160) {};
  \node[graphvertex] (v12) at (-1.740,-8.180) {};
  \node[graphvertex] (v13) at (-0.405,-8.670) {};
  \node[graphvertex] (v14) at (1.010,-8.531) {};
  \node[graphvertex] (v15) at (2.226,-7.793) {};
  \draw[graphedge] (v0) -- (v1);
  \node[edgelabel] at (-1.405,0.971) {0};
  \draw[graphedge] (v0) -- (v2);
  \node[edgelabel] at (0.160,1.100) {1};
  \draw[graphedge] (v0) -- (v3);
  \node[edgelabel] at (1.595,1.229) {2};
  \draw[graphedge] (v1) -- (v2);
  \node[edgelabel] at (-1.500,0.160) {3};
  \draw[graphedge] (v2) -- (v3);
  \node[edgelabel] at (1.500,0.160) {4};
  \draw[graphedge] (v2) -- (v4);
  \node[edgelabel] at (0.160,-1.100) {5};
  \draw[graphedge] (v4) -- (v1);
  \node[edgelabel] at (-1.595,-1.229) {6};
  \draw[graphedge] (v1) -- (v5);
  \node[edgelabel] at (-2.840,-1.100) {7};
  \draw[graphedge] (v4) -- (v5);
  \node[edgelabel] at (-1.500,-2.360) {8};
  \draw[graphedge] (v4) -- (v6);
  \node[edgelabel] at (0.160,-3.300) {9};
  \draw[graphedge] (v6) -- (v5);
  \node[edgelabel] at (-1.595,-3.429) {10};
  \draw[graphedge] (v5) -- (v7);
  \node[edgelabel] at (-2.840,-3.300) {11};
  \draw[graphedge] (v6) -- (v7);
  \node[edgelabel] at (-1.500,-4.560) {12};
  \draw[graphedge] (v6) -- (v8);
  \node[edgelabel] at (0.160,-5.500) {13};
  \draw[graphedge] (v8) -- (v3);
  \node[edgelabel] at (1.354,-3.234) {14};
  \draw[graphedge] (v3) -- (v9);
  \node[edgelabel] at (3.160,-3.300) {15};
  \draw[graphedge] (v8) -- (v9);
  \node[edgelabel] at (1.500,-6.440) {16};
  \draw[graphedge] (v7) -- (v10);
  \node[edgelabel] at (-2.931,-5.126) {17};
  \draw[graphedge] (v8) -- (v10);
  \node[edgelabel] at (-1.629,-6.361) {18};
  \draw[graphedge] (v10) -- (v11);
  \node[edgelabel] at (-2.804,-6.435) {19};
  \draw[graphedge] (v8) -- (v11);
  \node[edgelabel] at (-1.333,-7.037) {20};
  \draw[graphedge] (v11) -- (v12);
  \node[edgelabel] at (-2.120,-7.558) {21};
  \draw[graphedge] (v8) -- (v12);
  \node[edgelabel] at (-0.762,-7.508) {22};
  \draw[graphedge] (v12) -- (v13);
  \node[edgelabel] at (-1.017,-8.275) {23};
  \draw[graphedge] (v8) -- (v13);
  \node[edgelabel] at (-0.045,-7.665) {24};
  \draw[graphedge] (v13) -- (v14);
  \node[edgelabel] at (0.287,-8.441) {25};
  \draw[graphedge] (v8) -- (v14);
  \node[edgelabel] at (0.647,-7.491) {26};
  \draw[graphedge] (v14) -- (v15);
  \node[edgelabel] at (1.535,-8.025) {27};
  \draw[graphedge] (v8) -- (v15);
  \node[edgelabel] at (1.188,-7.055) {28};
  \draw[graphedge] (v15) -- (v9);
  \node[edgelabel] at (2.479,-7.109) {29};
\end{tikzpicture}}
\caption*{Close the arc}
\end{minipage}
\caption{Step-by-step construction of $V_{15}(\LeftChoice,\LeftChoice,\RightChoice)$.}
\label{fig:V15-construction}
\end{figure}

We now record some simple structural properties of the graphs constructed above.

\begin{proposition}\label{prop:VN-properties}
Let $S^{\mathrm{op}}$ be the binary sequence obtained from $S$ by interchanging $\LeftChoice$ and $\RightChoice$ in every entry.
Then the following statements hold.
\begin{enumerate}[itemsep=0.35em, topsep=0.35em]
\item \label{it:lr_iso} The graphs $V_N(S)$ and $V_N(S^{\mathrm{op}})$ are isomorphic. 
Likewise, $U_N(S)$ and $U_N(S^{\mathrm{op}})$ are isomorphic. 
If $N$ is odd, then the isomorphism has an even edge permutation.

\item \label{it:V_is_wheel} The graph $V_N(\emptyset)$ is the wheel graph $W_N$.

\item \label{it:V_low_valence} For a binary sequence $S$ of length $(N-3)/2$, every vertex of $V_N(S)$ has valence at most $4$.
\end{enumerate}
\end{proposition}

\begin{proof}

For item~\ref{it:lr_iso}, the isomorphism is given by reversing the order in which we add the edges in Steps~\ref{it:seed} and \ref{it:arc}. Keep in mind that a permutation reversing the order of $n$ items is even if $n \equiv 0,1 \pmod{4}$ and odd otherwise.
In Step~\ref{it:seed}, we add $5$ edges, which reverses with two vertex transpositions.
In Step~\ref{it:arc}, we add $2N-4k-5$ edges. For $N=2m+1$, we have
\[
2N-4k-5 = 4m-4k-3 \equiv 1 \pmod{4},
\]
so this reversal is even.
 
For $V_N(\emptyset)$, we skip from Step \ref{it:seed} directly to Step \ref{it:arc}. This results in a wheel graph where vertex $2$ is the centre hub vertex. 

Finally, we note that for any $V_N(S)$, the only vertex which may be more than $4$-valent is the vertex denoted C in Step \ref{it:arc}.
For a sequence of length $(N-3)/2$, we do not have any extra edges to attach to this vertex.

\end{proof}

\subsection*{Examples of $V_N$ and $U_N$}

We now illustrate some examples of graphs in the families $V_N(S)$ and $U_N(S)$ in Figure~\ref{fig:VN-UN-examples}.

\begin{figure}[H]
\centering
\begin{minipage}{0.48\linewidth}
\centering
\resizebox{0.50\linewidth}{!}{
}
\caption*{$V_9(\emptyset)$}
\end{minipage}\hfill
\begin{minipage}{0.48\linewidth}
\centering
\resizebox{0.50\linewidth}{!}{%
}
\caption*{$U_9(\emptyset)$}
\end{minipage}
\end{figure}

\begin{figure}[H]
\centering
\captionsetup{font=scriptsize}
\vspace{-0.8em}
\begin{minipage}{0.32\linewidth}
\centering
\resizebox{0.60\linewidth}{!}{%
}
\caption*{$V_9(\LeftChoice,\RightChoice,\LeftChoice)$}
\end{minipage}\hfill
\begin{minipage}{0.32\linewidth}
\centering
\resizebox{0.60\linewidth}{!}{%
}
\caption*{$V_9(\LeftChoice,\LeftChoice,\RightChoice)$}
\end{minipage}\hfill
\begin{minipage}{0.32\linewidth}
\centering
\resizebox{0.60\linewidth}{!}{%
}
\caption*{$V_9(\LeftChoice,\LeftChoice,\LeftChoice)$}
\end{minipage}
\vspace{-0.35em}

\par\noindent
\begin{minipage}{0.32\linewidth}
\centering
\resizebox{0.60\linewidth}{!}{%
}
\caption*{$V_{13}(\LeftChoice,\RightChoice)$}
\end{minipage}\hfill
\begin{minipage}{0.32\linewidth}
\centering
\resizebox{0.60\linewidth}{!}{%
}
\caption*{$U_{13}(\LeftChoice,\RightChoice)$}
\end{minipage}\hfill
\begin{minipage}{0.32\linewidth}
\centering
\resizebox{0.60\linewidth}{!}{%
}
\caption*{$V_{13}(\LeftChoice,\RightChoice,\LeftChoice)$}
\end{minipage}
\caption{Additional examples in the families $V_N(S)$ and $U_N(S)$.}
\label{fig:VN-UN-examples}
\end{figure}

The families $V_N(S)$ and $U_N(S)$ constructed in this way will be used to produce explicit low-valence representatives for the wheel classes.

\section[Differential of the family U\_N(S)]{Differential of \texorpdfstring{$U_N(S)$}{U_N(S)}}

Let us now consider the edge contraction differential on $U_N(S)$. 

The aim of this section is to prove the following lemma, which forms the core combinatorial step in the proof of the main theorem.
\begin{lemma}\label{lem:differential-UN}
Let $\mathcal{S}_k$ denote the set of all binary sequences of choices $\{\LeftChoice, \RightChoice\}$ of length $k$.
Then, for $N$ odd and $k\le(N-5)/2$, we have
\[
 d \left(\sum_{S \in \mathcal{S}_k} U_N(S) \right)
 = \sum_{T \in \mathcal{S}_{k+1}} V_N(T) - \sum_{S \in \mathcal{S}_k} V_N(S).
\]
\end{lemma}

\begin{proof}
For each $U_N(S)$ there are only a handful of edges that may be contracted without creating a double edge, which would kill the resulting graph in the even graph complex.

The edges that may be contracted are the following.
\begin{enumerate}[itemsep=0.35em, topsep=0.35em]
\item The edge labelled $4k+5$, created in Step~\ref{it:U}. The resulting graph is identical to $V_N(S)$. The sign of this contraction is
\[
(-1)^{4k+5+2N} = -1.
\]

\item The edge labelled $4k+6$, created in Step~\ref{it:left_arc}. This creates a graph that is identical to $V_N(S,\LeftChoice)$. The sign of this contraction is $+1$.

\item The edge labelled $2N$, created in Step~\ref{it:right_arc}. This creates a graph that is isomorphic to $V_N(S,\RightChoice)$. The sign of this contraction is $+1$.

The isomorphism is given by swapping the edges labelled $2N-1$ and $2N-3$ and by the three insertions
\[
2N-3 \mapsto 4(k+1)+1,
\]
\[
2N-2 \mapsto 4(k+1)+2,
\]
\[
2N-1 \mapsto 4(k+1)+3.
\]
All these four permutations carry an odd sign, so the composed sign is even.

\item Finally, some edges $4i+1$, for $i=2,\dots,k$, created in Step~\ref{it:spine}, may be contracted without creating a double edge. Let us denote such a graph by $\overline{U}_N^i(S)$.
\end{enumerate}

Consider
\[
\overline{U}_N^i(s_1,\dots,s_{i-1},s_i,\dots,s_k)
\]
for a sequence $S=(s_1,\dots,s_k)$. If $s_i=s_{i-1}$, then $\overline{U}_N^i(S)$ vanishes because it contains a double edge.

Otherwise, the graphs
\[
\overline{U}_N^i(s_1,\dots,s_{i-2},\LeftChoice,\RightChoice,\dots,s_k)
\]
and
\[
\overline{U}_N^i(s_1,\dots,s_{i-2},\RightChoice,\LeftChoice,\dots,s_k)
\]
are isomorphic, where the isomorphism is given by the three edge transpositions coming from \ref{it:triangle1}, \ref{it:triangle2}, and \ref{it:triangle3} of Step~\ref{it:iterate} in the construction list above. Therefore,
\[
\overline{U}_N^i(s_1,\dots,s_{i-1},s_i,\dots,s_k)
=
-\overline{U}_N^i(s_1,\dots,s_{i-2},s_i,s_{i-1},s_{i+1},\dots,s_k).
\]
Hence these terms cancel pairwise in the sum over all sequences. It follows that
\[
\begin{aligned}
\sum_{S \in \mathcal{S}_k} d\bigl(U_N(S)\bigr)
&= \sum_{S \in \mathcal{S}_k} \Bigl( V_N(S,\LeftChoice) + V_N(S,\RightChoice) - V_N(S) \Bigr)\\
&= \sum_{T \in \mathcal{S}_{k+1}} V_N(T) - \sum_{S \in \mathcal{S}_k} V_N(S).
\end{aligned}
\]

\end{proof}

\section{Proof of the Main Theorem}
We now state the explicit form of the main theorem, in which $U_{2m+1}$ is defined.

\begin{theorem}\label{thm:main-reduction}\label{thm:explicit-representatives}
For $k\geq 1$, let $\mathcal{S}_k^{\LeftChoice}\subset \mathcal{S}_k$ denote the subset of binary sequences of length $k$ whose first entry is $\LeftChoice$.
For each $m\geq 2$, define
\[
U_{2m+1}^{(0)}:=U_{2m+1}(\emptyset),
\]
and
\[
U_{2m+1}^{(k)}:=2\sum_{S\in \mathcal{S}_k^{\LeftChoice}} U_{2m+1}(S),
\qquad k=1,\dots,m-2,
\]
and
\[
U_{2m+1}:=\sum_{k=0}^{m-2} U_{2m+1}^{(k)}
=
U_{2m+1}(\emptyset)+2\sum_{k=1}^{m-2}\sum_{S\in \mathcal{S}_k^{\LeftChoice}} U_{2m+1}(S).
\]
Then
\[
W_{2m+1}+d(U_{2m+1})=2\sum_{S\in \mathcal{S}_{m-1}^{\LeftChoice}}V_{2m+1}(S).
\]
\end{theorem}

By Proposition~\ref{prop:VN-properties}, item~\ref{it:V_low_valence}, every graph on the right-hand side has only $3$- and $4$-valent vertices.

\begin{proof}
By Proposition~\ref{prop:VN-properties}, item~\ref{it:lr_iso},
\[
U_{2m+1}=\sum_{k=0}^{m-2}\sum_{S\in \mathcal{S}_k}U_{2m+1}(S).
\]
By Lemma~\ref{lem:differential-UN},
\[
d(U_{2m+1})
=
\sum_{k=0}^{m-2}
\left(
\sum_{T\in \mathcal{S}_{k+1}}V_{2m+1}(T)-\sum_{S\in \mathcal{S}_k}V_{2m+1}(S)
\right)
=
\sum_{S\in \mathcal{S}_{m-1}}V_{2m+1}(S)-V_{2m+1}(\emptyset).
\]
By Proposition~\ref{prop:VN-properties}, item~\ref{it:V_is_wheel}, the graph $V_{2m+1}(\emptyset)$ is the wheel graph $W_{2m+1}$. Hence
\[
W_{2m+1}+d(U_{2m+1})=\sum_{S\in \mathcal{S}_{m-1}}V_{2m+1}(S).
\]
Again by Proposition~\ref{prop:VN-properties}, item~\ref{it:lr_iso},
\[
\sum_{S\in \mathcal{S}_{m-1}}V_{2m+1}(S)
=
2\sum_{S\in \mathcal{S}_{m-1}^{\LeftChoice}}V_{2m+1}(S).
\]
This proves the theorem.
\end{proof}

\section{Examples}
Finally we present visualisations of the explicit $3$- and $4$-valent representatives of the homology classes $W_5$, $W_7$, $W_9$, and $W_{11}$ provided by Theorem~\ref{thm:explicit-representatives}:

\begin{figure}[H]
\centering
\begin{minipage}{0.18\linewidth}
\centering
$W_5 \simeq$
\end{minipage}
\begin{minipage}{0.24\linewidth}
\centering
\resizebox{0.9\linewidth}{!}{
}
\caption*{$V_5(\LeftChoice)$}
\end{minipage}
\end{figure}

\begin{figure}[H]
\centering
\begin{minipage}{0.16\linewidth}
\centering
$W_7 \simeq$
\end{minipage}
\begin{minipage}{0.22\linewidth}
\centering
\resizebox{0.9\linewidth}{!}{%
}
\caption*{$V_7(LL)$}
\end{minipage}
\begin{minipage}{0.035\linewidth}
\centering
$+$
\end{minipage}
\begin{minipage}{0.22\linewidth}
\centering
\resizebox{0.9\linewidth}{!}{%
}
\caption*{$V_7(LR)$}
\end{minipage}
\end{figure}

\vspace{-1.0em}

\begin{figure}[H]
\centering
\begin{minipage}{0.12\linewidth}
\centering
$W_9 \simeq$
\end{minipage}
\begin{minipage}{0.18\linewidth}
\centering
\resizebox{0.9\linewidth}{!}{%
}
\caption*{$V_9(LLL)$}
\end{minipage}
\begin{minipage}{0.02\linewidth}
\centering
$+$
\end{minipage}
\begin{minipage}{0.18\linewidth}
\centering
\resizebox{0.9\linewidth}{!}{%
}
\caption*{$V_9(LLR)$}
\end{minipage}
\begin{minipage}{0.02\linewidth}
\centering
$+$
\end{minipage}
\begin{minipage}{0.18\linewidth}
\centering
\resizebox{0.9\linewidth}{!}{%
}
\caption*{$V_9(LRL)$}
\end{minipage}
\begin{minipage}{0.02\linewidth}
\centering
$+$
\end{minipage}
\begin{minipage}{0.18\linewidth}
\centering
\resizebox{0.9\linewidth}{!}{%
}
\caption*{$V_9(LRR)$}
\end{minipage}
\end{figure}

\vspace{-1.1em}

\begin{figure}[H]
\centering
\begin{minipage}{0.12\linewidth}
\centering
$W_{11} \simeq$
\end{minipage}
\begin{minipage}{0.18\linewidth}
\centering
\resizebox{0.9\linewidth}{!}{%
}
\caption*{$V_{11}(LLLL)$}
\end{minipage}
\begin{minipage}{0.02\linewidth}
\centering
$+$
\end{minipage}
\begin{minipage}{0.18\linewidth}
\centering
\resizebox{0.9\linewidth}{!}{%
}
\caption*{$V_{11}(LLLR)$}
\end{minipage}
\begin{minipage}{0.02\linewidth}
\centering
$+$
\end{minipage}
\begin{minipage}{0.18\linewidth}
\centering
\resizebox{0.9\linewidth}{!}{%
}
\caption*{$V_{11}(LLRL)$}
\end{minipage}
\begin{minipage}{0.02\linewidth}
\centering
$+$
\end{minipage}
\begin{minipage}{0.18\linewidth}
\centering
\resizebox{0.9\linewidth}{!}{%
}
\caption*{$V_{11}(LLRR)$}
\end{minipage}

\vspace{-0.1em}

\begin{minipage}{0.12\linewidth}
\centering
$+$
\end{minipage}
\begin{minipage}{0.18\linewidth}
\centering
\resizebox{0.9\linewidth}{!}{%
}
\caption*{$V_{11}(LRLL)$}
\end{minipage}
\begin{minipage}{0.02\linewidth}
\centering
$+$
\end{minipage}
\begin{minipage}{0.18\linewidth}
\centering
\resizebox{0.9\linewidth}{!}{%
}
\caption*{$V_{11}(LRLR)$}
\end{minipage}
\begin{minipage}{0.02\linewidth}
\centering
$+$
\end{minipage}
\begin{minipage}{0.18\linewidth}
\centering
\resizebox{0.9\linewidth}{!}{%
}
\caption*{$V_{11}(LRRL)$}
\end{minipage}
\begin{minipage}{0.02\linewidth}
\centering
$+$
\end{minipage}
\begin{minipage}{0.18\linewidth}
\centering
\resizebox{0.9\linewidth}{!}{%
}
\caption*{$V_{11}(LRRR)$}
\end{minipage}
\end{figure}

\end{document}